\documentclass[11pt]{amsart}

\usepackage{amsfonts,amssymb,amsmath,amsthm,latexsym,amscd}
\usepackage{epsfig,verbatim}
\usepackage{graphics,textcomp, hyperref}
\usepackage{graphicx}
\usepackage{url}

\theoremstyle{definition}
\newtheorem{defn}{Definition}[section]
\newtheorem{rem}[defn]{Remark}

\theoremstyle{plain}
\newtheorem{thm}[defn]{Theorem}

\newtheorem{lemma}[defn]{Lemma}

\newtheorem{qn}[defn]{Question}
\newtheorem{res}[defn]{Result}

\numberwithin{equation}{section}

\newcommand{\Fu}{\mathcal{C}}

\begin{document}

\title[The Fujiki class and positive degree maps]{The Fujiki class
and positive degree maps}

\author[G. Bharali]{Gautam Bharali}
\address{Department of Mathematics, Indian Institute of Science, 
Bangalore 560012, India}
\email{bharali@math.iisc.ernet.in}

\author[I. Biswas]{Indranil Biswas}
\address{School of Mathematics, Tata Institute of Fundamental
Research, Homi Bhabha Road, Bombay 400005, India}
\email{indranil@math.tifr.res.in}

\author[M. Mj]{Mahan Mj}
\address{Department of Mathematics, RKM Vivekananda University, P.O. 
Belur Math, Howrah 711202, India}
\email{mahan@rkmvu.ac.in}



\thanks{GB is supported by a UGC Centre for Advanced Study grant. IB and MJ are supported
by J.C. Bose Fellowships}

\begin{abstract}
We  show that a map between complex-analytic manifolds, at least one of which
is in the Fujiki class, is a biholomorphism under a natural condition on the second cohomologies.
We use this to establish that, with mild restrictions, a certain relation of ``domination'' introduced by Gromov
is in fact a partial order.  
\end{abstract}

\maketitle

\section{Introduction}\label{S:intro}
Gromov introduced a relation of {\em domination} between smooth 
closed manifolds of a fixed dimension by declaring $M \geq N$ if there is a smooth map of
positive degree from $M$ to $N$ (see Section~\ref{gpo} for details).
This relation can be made more restrictive by demanding that the map be of degree one. The main aim
of this paper is to describe this relation in the context of smooth projective varieties and K{\"a}hler 
manifolds, and show that this relation is in fact a partial order on certain natural classes of complex
manifolds of a fixed dimension, e.g.,
\begin{enumerate}
\item on the class of smooth projective varieties of general type, 
\item on the class of smooth projective varieties that are Kobayashi hyperbolic.
\end{enumerate}
One can add to this list; see Theorem~\ref{gro-po}. 
Gromov introduced the relation ``$\geq$'' in the context of constant negative sectional curvature. Theorem \ref{gro-po}
may be regarded as an attempt to translate this into the context of negative curvature in complex algebraic geometry.
To the best of our knowledge, this relation
has not been dealt with in the context of complex-analytic manifolds or smooth varieties.
\smallskip

A first step towards the above project is to try to address the following loosely-worded (but related)
question: Let $X$ and $Y$ be two compact complex manifolds and let $f: Y\,\longrightarrow\,X$ be a
surjective holomorphic map. If one of these manifolds is K{\"a}hler, then can one deduce further
information about the other manifold or the map $f$\,? (In this article, it is implicitly assumed that a
manifold is connected.)
\smallskip

It turns out that if $Y$ is a K{\"a}hler manifold, then one can say quite a lot about $X$. We begin
with the following definition:

\begin{defn}[Fujiki, \cite{fu}]
A reduced compact complex space $X$ is said to belong 
to the {\em Fujiki class $\Fu$} if it is a meromorphic  image of a compact K{\"a}hler space.
\end{defn}

\noindent{In particular, any manifold $X$ as in the question above, with $Y$ a K{\"a}hler manifold,
is in the Fujiki class $\Fu$ (and has all the properties known about the elements in this class).   
However, without any additional conditions, we cannot even infer that $X$ is
K{\"a}hler: any non-projective Moishezon manifold is in the class $\Fu$ but is non-K{\"a}hler.
In contrast, Varouchas has shown \cite{varouchas:scvK} that, with $Y$, $X$ and $f$ as in the question
above and $Y$ K{\"a}hler, if every fiber of $f$ has the same dimension, then $X$ is K{\"a}hler.
This tactic\,---\,i.e., imposing some condition on the map $f$ in the above question\,---\,turns out
to be quite useful. It gives us our first result, which is a key tool in proving Theorem~\ref{gro-po}.
\smallskip
  
We ought to mention that Varouchas \cite[Section\,IV.3]{Va} has shown that a compact complex manifold belongs to
Fujiki's class $\Fu$ if and only if it is {\em bimeromorphic} to a compact K{\"a}hler manifold. It is the
latter property that we shall use in our proof of the following result.  

\begin{thm}\label{thm2}
Let $X$ and $Y$ be compact connected complex manifolds satisfying
$$
\dim X\,=\,\dim Y \quad \text{and} \quad
\dim H^2(X,\,{\mathbb Q})\,=\,\dim H^2(Y,\,{\mathbb Q}).
$$
Let $\varphi : Y\,\longrightarrow\, X$ be a surjective holomorphic map of degree one. If at least one of
$X$ and $Y$ is in the Fujiki class $\Fu$, then $\varphi$ is a biholomorphism.
\end{thm}

Theorem \ref{thm2} is a variation of a result of \cite{BB} which is recalled
in Section \ref{sec3}. We shall use the Hironaka elimination of indeterminacies (cf. Lemma \ref{lem1}) as a tool.
A straightforward consequence is the fact (known to experts)  that the class of fundamental groups
of compact manifolds in the Fujiki class coincide with the class of K\"ahler groups (see the last paragraph
of \cite{Ar}).
\smallskip

Theorem~\ref{thm2} is used in Section~\ref{gpo} to show that the Gromov relation introduced at the
beginning of this paper is in fact a partial order under certain natural hypotheses. 
In the notation of Section~\ref{gpo}, we have the following key application of Theorem~\ref{thm2}, and the
principal step to the main result of Section~\ref{gpo}:

\begin{thm}\label{P:cor3}
Let $X$ and $Y$ be compact connected complex manifolds with $\dim X\,=\, \dim Y$
such that at least one of $X, Y$ belongs to the Fujiki class $\Fu$.
Further, suppose that $X \geq Y$ and $Y\geq X$. If $X, Y$ are not biholomorphic,
then $X$ admits a self-endomorphism of degree greater than one.
\end{thm}

\section{A criterion for biholomorphism: The proof of
Theorem~\ref{thm2}}\label{sec3}
We begin with the following lemma, which is a form of the Hironaka
elimination of indeterminacies.

\begin{lemma}\label{lem1}
Let $X$ be a compact connected complex manifold in the Fujiki class $\Fu$.
Then there exists a pair $(Y, f)$, where $Y$ is a compact connected K{\"a}hler manifold with
$\dim Y\,=\, \dim X$, and
$$
f\,:\,Y\,\longrightarrow\,X
$$
is a surjective {\em holomorphic} map of degree one.
\end{lemma}

\begin{proof}
Let $d$ be the complex dimension of $X$.
A theorem of Varouchas in \cite{Va} says that $X$ is bimeromorphic to a compact
connected K{\"a}hler manifold of complex dimension $d$\,---\,see \cite[p.\,31, Theorem~10]{Ba} for
a short proof). Let
$$
\phi\,:\,Z\,\dasharrow\,X
$$
denote such a bimeromorphic map from a compact
K{\"a}hler manifold $Z$ of dimension $d$.
\smallskip

The \textit{elimination of indeterminacies} says that there is a finite sequence of
holomorphic maps
$$
Z_n\,\stackrel{f_n}{\longrightarrow}\,Z_{n-1}\,
\stackrel{f_{n-1}}{\longrightarrow}\,Z_{n-2}\,\stackrel{f_{n-2}}{\longrightarrow}\,
\cdots \,\stackrel{f_2}{\longrightarrow}\,Z_1\,\stackrel{f_1}{\longrightarrow}\,Z_0
\,=\,Z
$$
such that each $(Z_i, f_i)$, $1\leq i\leq n$, is a blow-up of a smooth complex
submanifold of $Z_{i-1}$, and the bimeromorphic map
$$
\phi\circ f_1\circ \cdots \circ f_n\,:\,Z_n \,\dasharrow\,X
$$
extends to a holomorphic map
$$
\widetilde{\phi}\,:\,Z_n\,\longrightarrow\,X;
$$
see \cite{Hi2} and \cite{Hi1} by Hironaka. We refer the reader to  \cite[p.\,539, \S\,1.2.4]{AKMW} for
some explanation of how the above process can be carried out\,---\,ensuring, especially, that each
successive blow-up is along a smooth center\,---\,using \cite{Hi2}. (This process works in the analytic
case as well as in the algebraic; the case of complex-analytic manifolds is addressed in the last
two paragraphs of \cite[\S\,1.2.4]{AKMW}.) The blow-up of a smooth submanifold
of a K{\"a}hler manifold is K{\"a}hler \cite[p.\,202, Th{\'e}or{\`e}me~II.6]{Bl}. Since
$Z_0$ is K{\"a}hler, we conclude that all $Z_i$ are K{\"a}hler.
\smallskip

We set $Y:=Z_n$ and $f:=\widetilde{\phi}$ to obtain the desired pair $(Y,f)$.
\end{proof}

Let $X$ and $Y$ be compact connected complex manifolds with $\dim X= \dim Y$,
and let
$$
\varphi\,:\,Y\,\longrightarrow\, X
$$
be a surjective holomorphic map of degree one. In \cite{BB} the following was proved: if the underlying
real manifolds for $X$ and $Y$ are diffeomorphic, and also
$$
\dim H^1(X,\,{\mathcal O}_X)\,=\,\dim H^1(Y,\,{\mathcal O}_Y)\, ,
$$
then $\varphi$ is a biholomorphism.
\smallskip

Theorem~\ref{thm2} is a variation of the above result.

\begin{proof}[Proof of Theorem~\ref{thm2}]
The pullback homomorphisms of cohomologies
$$
\varphi^*_i\,:\,H^i(X,\,{\mathbb Q})\,\longrightarrow\,
H^i(Y,\,{\mathbb Q})
$$
are injective for all $i$. Therefore, from the given condition that $\dim H^2(X,\,
{\mathbb Q})\,=\,\dim H^2(Y,\,{\mathbb Q})$ it follows that the homomorphism
\begin{equation}\label{e1}
\varphi^*_2\,:\,H^2(X,\,{\mathbb Q})\,\longrightarrow\,
H^2(Y,\,{\mathbb Q})
\end{equation}
is an isomorphism.
\smallskip

The complex dimension of $X$ will be denoted by $d$. Let
$$
\bigwedge\nolimits^d d\varphi\,\in\,H^0\left(Y,\,\text{Hom}\big(
\bigwedge\nolimits^d T^{1,0}Y\,, \varphi^* \bigwedge\nolimits^d T^{1,0}X\big)\right)
$$
be the $d$-th exterior power of the differential of $\varphi$. The divisor for this homomorphism
$\bigwedge^d d\varphi$ will be denoted by $D$. The map $\varphi$ is a biholomorphism
if $D$ is the zero divisor.
\smallskip

Since the degree of $\varphi$ is one, the image $\varphi(D)$ is of complex codimension
at least two in $X$. Therefore, if
$$
c_D\, \in\, H_{2d-2}(Y,\,{\mathbb Q})
$$
is the class of $D$, then its image $\varphi_*(c_D)\in H_{2d-2}(X,\,{\mathbb Q})$
is zero. This implies that the Poincar{\'e} duality pairing of $c_D$ with
$\varphi^*(H^2(X,\,{\mathbb Q}))\,\subset\, H^2(Y,\, {\mathbb Q})$ vanishes identically.
Since $\varphi^*_2$ in \eqref{e1} is surjective, we now conclude that
\begin{equation}\label{e2}
c_D\,=\, 0\, .
\end{equation}

Let us first assume that $Y$ lies in the Fujiki class $\Fu$.
By Lemma \ref{lem1}, there is a compact connected K{\"a}hler manifold $Z$ of dimension $d$
and a surjective holomorphic map
$$
f\,:\,Z\,\longrightarrow\, Y
$$
of degree one. Consider the effective divisor
$$
\widetilde{D}\,:=\,f^{-1}(D)\,\subset\, Z\, .
$$
From \eqref{e2} we know that the class of $\widetilde{D}$ in
$H_{2d-2}(Z,\, {\mathbb Q})$ vanishes. Since $Z$ is K{\"a}hler, this implies that
$\widetilde{D}$ is the zero divisor. Hence $D$ is the zero divisor. Consequently,
$\varphi$ is a biholomorphism.
\smallskip

Now assume that $X$ lies in the Fujiki class $\Fu$. Therefore,
$X$ is bimeromorphic to a compact connected K\"ahler manifold $Z$ of dimension $d$
\cite{Va}; also see \cite[p.\,31, Theorem~10]{Ba}. Since $\varphi$ is a bimeromorphic
map from $Y$ to $X$, it follows that $Y$ is bimeromorphic to $Z$. Hence $Y$
lies in the Fujiki class $\Fu$. We have already shown that $\varphi$
is a biholomorphism if $Y$ lies in the Fujiki class $\Fu$.
\end{proof}
\smallskip

\section{Gromov Partial Order}\label{gpo}
In a lecture he gave at the Graduate Center CUNY 
in the spring of 1978, Gromov had introduced a notion of ``domination'' between smooth
manifolds as follows \cite{tol-pc, carltol}:
\begin{itemize}
 \item[{}] Let $X, Y$ be closed smooth $n$-manifolds. We say that $X \geq Y$
if there  is a smooth map of positive degree from $X$ to $Y$.
\end{itemize}
\smallskip
 
Gromov introduced this notion in the context of real hyperbolic manifolds. It is not
clear a priori whether ``$\geq$'' is in fact a partial order or not. We transfer this question to the context of projective
and K{\"a}hler manifolds and holomorphic maps between them.
We rephrase this as follows:

\begin{qn}\label{poqn}
Let $X, Y$ be compact projective
(respectively, K{\"a}hler) manifolds of complex dimension $n$. We say that $X \geq Y$ if there  is a 
holomorphic map of positive degree from $X$ to $Y$. We say that $X \geq_1 Y$ if there  is a 
holomorphic map of  degree one from $X$ to $Y$. \\
1) If $X\geq Y$ and $Y\geq X$, are $X$ and $Y$ biholomorphic\,?\\
2) If $X\geq_1 Y$ and $Y\geq_1 X$, are $X$ and $Y$ biholomorphic\,?
\end{qn}

\begin{rem}
Given the conditions on the manifolds $X, Y$ in the discussion above, any positive-degree
holomorphic map from one of them to the other is automatically surjective. Thus, whenever we
apply Theorem~\ref{thm2} to some positive-degree map in the proofs below, we will not
remark upon its surjectivity.
\end{rem}
   
As a consequence of Theorem~\ref{thm2} we have the following theorem, which
is a step towards answering Question~\ref{poqn}. Theorem~\ref{P:cor3}, stated
in the introduction forms a part of the following result, and Part\,(1) below provides
an answer to Question~\ref{poqn}(2).

\begin{thm}\label{prop4}
Let $X$ and $Y$ be compact connected complex manifolds with
$\dim X\,=\, \dim Y$ such that at least one of $X, Y$ belongs
to the Fujiki class $\Fu$. 
\begin{enumerate}
\item If $X \geq_1 Y$ and $Y\geq_1 X$, then $X$ and $Y$ are biholomorphic.

\item Assume that $X \geq Y$ and $Y\geq X$. If $X, Y$ are not
biholomorphic, then $X$ admits a self-endomorphism of degree greater than one.
\end{enumerate}
\end{thm}
\begin{proof}
We may take $X$ to be in the Fujiki class $\Fu$.  
If $f\,:\,X\,\longrightarrow\, Y$ and $g\,:\,Y\,\longrightarrow\, X$
are degree one maps, then $g\circ f$ is a holomorphic automorphism by Theorem~\ref{thm2}.
Therefore, $f$ is a biholomorphism, which proves Part\,(1).
\smallskip

We now consider Part\,(2). Since $X\geq Y$ (respectively, $Y\geq X$),
we have $b_2(X)\,\geq\, b_2(Y)$ (respectively, $b_2(Y)\,\geq\, b_2(X)$)
because, by the assumption of positivity of degree of the map from $X$ to $Y$ (respectively,
$Y$ to $X$), the pullback homomorphism of cohomologies is injective.
Therefore, $b_2(X)\, =\, b_2(Y)$. If $X$ is not biholomorphic to $Y$,
then by Theorem \ref{thm2} we conclude that the degree of any surjective
holomorphic map between $X$ and $Y$ is at least two. Now Part\,(2) now
follows by taking composition of two such maps $X\,\longrightarrow\, Y$
and $Y\,\longrightarrow\, X$.
\end{proof}

The following result summarizes standard facts about non-existence of
non-trivial self-endomorphisms: see \cite{beau}, and \cite{fujimoto} along with the references therein.

\begin{res}\label{noendo}
Let $X$ be a compact connected complex manifold, and let $X$ satisfy one of
the following:
\begin{enumerate}
 \item $X$ is a projective manifold of general type.
 \item $X$ is Kobayashi hyperbolic.
 \item $X$ is a rational homogeneous manifold of Picard number $1$ that is not biholomorphic to
 a complex projective space $\mathbb{P}^n$.
 \item $X$ is a smooth projective hypersurface of dimension greater than $1$ and of degree
 greater than $2$.
\end{enumerate}
Then, any self-endomorphism of $X$ of positive degree is an automorphism.
\end{res}

We combine Theorem~\ref{prop4} with Result~\ref{noendo} to deduce the following:

\begin{thm}\label{gro-po}
Let $X$ and $Y$ be compact connected projective manifolds with $\dim X = \dim Y$.
Suppose that $X \geq Y$ and $Y\geq X$. Also suppose that at least one of $X, Y$
belongs to one of the four classes listed in Result~\ref{noendo}. Then $X$ and $Y$ are biholomorphic.
\end{thm}
\smallskip

\section*{Acknowledgments}

\noindent{We thank N. Fakhruddin for helpful discussions. We also thank the referee of this work
for helpful suggestions on our exposition.}
%
\medskip


\end{document}